\documentclass[]{amsart}
\usepackage[utf8]{inputenc}
\usepackage{emptypage}
\usepackage{amsopn,amstext,amsbsy,amsmath,amscd,amsthm,amsfonts,systeme,amsfonts,mathtools,amssymb,siunitx}
\usepackage{mathtools}
\usepackage{a4wide}
\usepackage{mathrsfs}
\usepackage{amssymb}
\usepackage{verbatim}
\usepackage{enumerate}
\usepackage[shortlabels]{enumitem}
\usepackage{tikz}
\usepackage{mwe}
\usepackage{float}
\usepackage{subfig}
\usepackage{hyperref}
\usepackage{listings}
\usepackage{float}
\usepackage{todonotes}
\usepackage{graphicx}
\usepackage[T1]{fontenc}
\usepackage{amsmath}
\usepackage{latexsym}
\usepackage{graphics,epsfig}
\usepackage{color}
\usepackage{a4wide,mathrsfs}
\usepackage{multirow}

\DeclareMathOperator{\codim}{codim}
\DeclareMathOperator{\real}{Re}
\usepackage{hyperref}
\usepackage[style=phys,articletitle=true,biblabel=brackets,chaptertitle=false,pageranges=true,sorting=nty,backend=biber]{biblatex}
\theoremstyle{plain}
\newtheorem{theorem}        {\bf Theorem}[section]

\newtheorem{lemma}        [theorem]  {\bf Lemma}

\theoremstyle{definition}
\newtheorem{example}      [theorem]  {\bf Example}
\newtheorem{definition}   [theorem]  {\bf Definition}

\theoremstyle{remark}
\newtheorem{assumption}      [theorem]  {\bf Assumption}

\numberwithin{equation}{section}

\title{Embedded eigenvalues for asymptotically periodic ODE systems}
\author{Sara Maad Sasane and Wilhelm Treschow}
\date{\today}

\begin{document}

\begin{abstract}
We investigate the persistance of embedded eigenvalues under perturbations of a certain self-adjoint Schrödinger-type differential operator in $L^2(\mathbb{R};\mathbb{R}^n)$, with an asymptotically periodic potential. The studied perturbations are small and belong to a certain Banach space with a specified decay rate, in particular, a weighted space of continuous matrix valued functions. Our main result is that the set of perturbations for which the embedded eigenvalue persists forms a smooth manifold with a specified co-dimension. This is done using tools from Floquet theory, basic Banach space calculus, exponential dichotomies and their roughness properties, and Lyapunov-Schmidt reduction. A second result is provided, where under an extra assumption, it can be proved that the first result holds even when the space of perturbations is replaced by a much smaller space, as long as it contains a minimal subspace.
In the end, as a way of showing that the investigated setting exists, a concrete example is presented. The example itself relates to a problem from quantum mechanics and represents a system of electrons in an infinite one-dimensional crystal.
\end{abstract}

\maketitle

\section{Introduction}

\subsection{Background}
It is well known that eigenvalues separated from the rest of the spectrum persist under small perturbations \cite[p.~213]{kato}.
On the other hand,  eigenvalues embedded in the continuous spectrum behave very differently. Such eigenvalues typically disappear after adding an arbitrary small perturbation \cite{agmon}.
The study of embedded eigenvalues are relevant to problems from physics. In quantum mechanics for example, eigenvalues represent energy states of the underlying system, and if these eigenvalues are embedded in the continuous spectrum, they might be very sensitive to small perturbations of the potential \cite{alexia}. Embedded eigenvalues and their existence and persistence have previously been studied in \cite{herbst} and \cite{sharptrans}.

The goal of this paper is to study a perturbation problem of a self-adjoint operator on $L^2(\mathbb{R};\mathbb{R}^n)$, involving an asymptotically periodic potential, which has an embedded eigenvalue of multiplicity $1$. We prove that the set of perturbations for which the embedded eigenvalue persists forms a Banach manifold in the Banach space of perturbations. We also specify its co-dimension. This is a generalization of the results of the article \cite{alexia}, where the problem was studied in the asymptotically constant case.

\subsection{Problem setup}

Consider the operator 
\begin{equation}
\mathcal{L} = - \frac{d^2}{dx^2} + A(x),
\label{eq:opdef}
\end{equation}
on the Hilbert space $L^2(\mathbb{R}; \mathbb{R}^n)$, where $A: \mathbb{R} \to \mathbb{R}^{n\times n}$ is a continuous and symmetric matrix-valued function, and is commonly referred to as the potential. Additionally, the matrix $A(x)$ is assumed to be asymptotically periodic, meaning that there exists a matrix-valued function $A_p(x)$ and a corresponding $p > 0$ such that $A_p(x+p) = A_p(x)$ for $x \in \mathbb{R}$, and $\lvert A(x) - A_{p}(x) \rvert \to 0 \text{ as } \lvert x \rvert \to \infty$, where $\lvert \cdot \rvert$ denotes the finite dimensional max-norm.
\bigskip

The eigenvalue equation for the operator $\mathcal{L}$ is
\begin{equation}
\mathcal{L}\mathbf{u} = \lambda \mathbf{u}.
\label{eq:eigenveq}
\end{equation}

 The scalar $\lambda \in \mathbb{C}$ is said to be an eigenvalue of $\mathcal{L}$ if it satisfies \eqref{eq:eigenveq} for some non-trivial eigenfunction $\mathbf{u}\in L^2(\mathbb{R};\mathbb{R}^n)$. Or, equivalently, $\lambda\in \mathbb{C}$ is said to be an eigenvalue of $\mathcal{L}$ if $\ker (\mathcal{L} - \lambda I)$ is non-trivial. In that case we say that $\lambda\in\sigma_p(\mathcal{L})$, where $\sigma_p(\mathcal{L})$ is called the point spectrum of $\mathcal{L}$. 
 
The continous spectrum $\sigma_c(\mathcal{L})$ of $\mathcal{L}$ consists of all $\lambda \in \mathbb{C}$ such that the resolvent operator $R_{\lambda}=(\mathcal{L}-\lambda I)^{-1}$ exists as an unbounded operator, and $\text{dom} R_{\lambda}$ is not closed. Since $\mathcal{L}$ is self-adjoint, its spectrum, $\sigma(\mathcal{L})$, is a subset of the real line and $\sigma(\mathcal{L}) = \sigma_p(\mathcal{L}) \cup \sigma_c(\mathcal{L})$ \cite[p.~37]{melvyn}, \cite[p.~90]{glaz}. Note that with these definitions, $\sigma_p(\mathcal{L})$ and $\sigma_c(\mathcal{L})$ may not be disjoint.

For further characterization of the spectrum of a self-adjoint operator, see \cite[p.~93]{glaz}. 

\bigskip

An embedded eigenvalue, $\lambda_0$, is an eigenvalue which also belongs to the continuous spectrum, i.e., $\lambda_0 \in \sigma_p(\mathcal{L})\cap\sigma_c(\mathcal{L})$.

\bigskip

For the type of operators that we consider, it is known that the continuous spectrum consists of the set of $\lambda \in \mathbb{R}$ for which there exist polynomially bounded solutions of the eigenvalue problem (\ref{eq:eigenveq}), that does not belong to $L^2(\mathbb{R};\mathbb{R}^n)$ \cite{sim}. These solutions are referred to as generalized eigenfunctions.

\bigskip

We want to analyze the persistence of the embedded eigenvalue $\lambda_0$ when a small perturbation $B$, which decays to zero with a certain decay rate, is added to the potential $A$. Therefore we consider the perturbed operator $\mathcal{L}_B=\mathcal{L}+B$, with $B$ in a certain Banach space, given by
\begin{equation*}
    (\mathcal{L}+B)\mathbf{u} = -\mathbf{u}'' + [A(x) + B(x)]\mathbf{u},
\end{equation*}
and wish to study the set of small perturbations for which the embedded eigenvalue $\lambda_0$ persists. By persistence we mean that there is an eigenvalue of the perturbed operator close to the original eigenvalue.

\subsection{Main result}
 
 We define the Banach space 
 
 $$X_{\beta} = \{B\in C(\mathbb{R},\mathbb{R}^{n\times n}); B(x)^T=B(x) \hspace{1mm} \forall x\in\mathbb{R} \text{ and } \lVert B \rVert_{X_{\beta}} = \sup_{x\in\mathbb{R}}\lvert B(x)\rvert (1+\lvert x\rvert)^{\beta} < \infty \}, \quad \beta >1.$$

 The perturbed eigenvalue equation is
\begin{equation}
    -\mathbf{u}'' + (A(x) + B(x))\mathbf{u} = \lambda \mathbf{u}, \quad B\in X_{\beta}.
    \label{eq:perteigeq}
\end{equation}
This can be written as a system of first order ODEs, taking $\mathbf{u} = \mathbf{u}_1$ and $\mathbf{u}' = \mathbf{u}_2$, obtaining
\begin{equation}
    U' = \mathcal{A}(x;\lambda,B)U,
    \label{eq:1stordperteq}
\end{equation}
where $U = (\mathbf{u}_1,\mathbf{u}_2)^T\in \mathbb{R}^{2n}$ and $$\mathcal{A}(x;\lambda,B) = \begin{bmatrix} 0 & I \\
A(x) + B(x) - \lambda I & 0
\end{bmatrix}.$$
We make the following assumptions.
 
\begin{assumption}\label{assume:1}
We assume that $\rVert A-A_p\lVert_{X_{\beta}}<\infty$, for some $\beta >1$.
\end{assumption}

\begin{assumption}\label{assume:2}
We assume that $\lambda_0$ is a simple eigenvalue of $\mathcal{L}$ and that  $1$ is not an eigenvalue of $\Phi(p)$, where $\Phi(x)$ is the fundamental matrix solution of the first order system of linear ODEs with coefficient matrix 
$$\begin{bmatrix}
    0 & I \\
    A_p(x) - \lambda_0I & 0
\end{bmatrix},$$
i.e., the matrix solution of the system satisfying $\Phi(0)=I$. 
\end{assumption}
Assumption \ref{assume:1} and \ref{assume:2} will be standing assumptions throughout the paper.

\bigskip

We denote by $2m$ the number of eigenvalues of the monodromy matrix of the system, $M=\Phi(p)$, with modulus $1$, counted with multiplicity, with $\Phi(x)$ as in Assumption \ref{assume:2}. In Section 2 we see that $2m$ is indeed an even number and $m$ an integer.

\bigskip

In Section 5 we provide an example to show that these assumptions can be met, so that the investigated setting exists.

\bigskip

We are now ready to state our main results.

\begin{theorem}\label{thm:main}
Let $\lambda_0$ be an eigenvalue of the unperturbed operator $\mathcal{L}$. Assume that Assumption \ref{assume:1} and Assumption \ref{assume:2} holds. Further, let $$\mathscr{S}_{\varepsilon} = \{B\in X_{\beta}| \exists \lambda\in (\lambda_0-\varepsilon, \lambda_0 + \varepsilon); \text{ such that }\lambda \text{ is an eigenvalue of } \mathcal{L}+B\}.$$
Then there exists an $\varepsilon > 0 $ and a neighbourhood $\mathcal{U}$ of $0\in X_{\beta}$, such that $\mathscr{S}_{\varepsilon} \cap \mathcal{U}$ is a manifold of codimension $2m$ in $X_{\beta}$.
\end{theorem}
Consider now the operator $\mathcal{L}$ in \eqref{eq:opdef}, under the additional assumption that $A$ is a diagonal matrix. In this special case, we want to study the perturbed operator $\mathcal{L}+B,$ with perturbations $B$ belonging to a subspace $Y_\beta\subset X_\beta$ containing a certain minimal subspace $T_\beta$ defined below. As $A$ is diagonal and the embedded eigenvalue $\lambda_0$ of the unperturbed operator is simple by assumption, only one of the entries of $\bf{u_*}$ is nonzero. Indeed, as $A$ is diagonal, the eigenvalue equation \eqref{eq:eigenveq} decouples, and there exists exactly one entry $u_j$ of the eigenfunction $\mathbf{u}_*$ with $j\in \{1,...,n\}$ such that $-u_j'' + a_{jj}u_j=\lambda_0u_j$ and $u_j\neq 0$.

To simplify notation, we assume without loss of generality that $j=1$. Then we define the minimal subspace $T_\beta$ as the subspace of $X_\beta$ for which $b_{ii}=0$ for $i=1,\dots,n$ and $b_{ij}=0$ unless $i=1$ or $j=1$. Hence, $T_\beta\subset X_\beta$ is the set of matrices of the type
$$B=\begin{bmatrix}
    0 & b_{12} & \cdots & \cdots & b_{1n} \\
    b_{12} & 0 & 0 & \cdots & 0 \\
    \vdots & 0 & \ddots \\
    \vdots & \vdots & & \ddots \\
    b_{1n} & 0 & & & 0
\end{bmatrix}.$$
\begin{assumption}\label{assume:3}
    We assume that the matrix $A$ is diagonal and that, without loss of generality, the index for which $u_j\neq 0$ is $j=1$. We assume also that the perturbation $B$ is on the special form with zeros on the diagonal.  
\end{assumption}
Assumption \ref{assume:3} is not a standing assumption and will be assumed only when explicitly stated.
\begin{theorem}\label{thm:second}
Let $\lambda_0$ be an eigenvalue of the unperturbed operator $\mathcal{L}$. Assume that Assumption \ref{assume:1}, Assumption \ref{assume:2} and \ref{assume:3} holds. Further, let $$\mathscr{T}_{\varepsilon} = \{B\in Y_{\beta}| \exists \lambda\in (\lambda_0-\varepsilon, \lambda_0 + \varepsilon); \text{ such that }\lambda \text{ is an eigenvalue of } \mathcal{L}+B\}.$$
Then there exists an $\varepsilon > 0 $ and a neighbourhood $\mathcal{U}$ of $0\in Y_{\beta}$, such that $\mathscr{T}_{\varepsilon} \cap \mathcal{U}$ is a manifold of codimension $2m$ in $Y_{\beta}$.
\end{theorem}
Theorem \ref{thm:second} says that in the special case where $A$ is diagonal, the conclusion of Theorem \ref{thm:main} holds even when the space of perturbations is replaced by a much smaller space than $X_{\beta}$, as long as it contains the specified minimal subspace, $T_{\beta}$.

These theorems gives us the desired result in form of a manifold of perturbations which do not remove the embedded eigenvalue when added to the original operator.
The methods we use combine those in \cite{alexia} with methods from Floquet theory.
 
\section{Preliminary results}
This section contains results about the asymptotic behaviour of solutions of our system, as well as a description of how to transform our system to a setting more similar to the one in \cite{alexia}.

\subsection{System at infinity}
Since  $A(x) - A_p(x)\to 0$ and $B(x) \to 0$ as $\lvert x \rvert \to \infty$, by replacing the coefficient matrix by the periodic background potential, we obtain the system at infinity 
\begin{equation}
    U' = \mathcal{A}_p(x;\lambda)U,
    \label{eq:systematinf}
\end{equation}
with $$\mathcal{A}_p(x;\lambda) = \begin{bmatrix} 0 & I \\
A_p(x) - \lambda I & 0
\end{bmatrix}.$$
Clearly, $\mathcal{A}_p(x;\lambda)$ is periodic with period $p$.

We can now decompose our original coefficient matrix, $\mathcal{A}(x;\lambda,B)$, into two, one with the system at infinity matrix, and one perturbation matrix, and write \eqref{eq:1stordperteq} as
\begin{equation}
    U' = (\mathcal{A}_p(x;\lambda_0) + L(x;\lambda,B))U,
    \label{eq:firstorderpert}
\end{equation}
where $$L(x;\lambda,B) = \begin{bmatrix} 0 & 0 \\
A(x)-A_p(x) + B(x) + (\lambda_0 - \lambda) I & 0
\end{bmatrix}$$
is considered a perturbation. 

\bigskip

We now recall a famous theorem by Floquet.
\begin{theorem}[Floquet's theorem]
Let $C(x)$ be a continuous periodic matrix-valued function with period $p$, and let $\Phi(x)$ be the fundamental matrix of the system $$\mathbf{y}'=C(x)\mathbf{y}.$$ Then there exists a non-singular continuously differentiable matrix-valued function $G(x)$ with period p, and a constant, possibly complex, matrix R such that 
$$\Phi(x) = G(x)e^{Rx}, \text{ for all } x\in \mathbb{R}.$$
\end{theorem}
For a proof, see \cite[p.~92]{teschl}.

The system at infinity \eqref{eq:systematinf}, clearly satisfies Floquet's theorem with a periodic and continuously differentiable matrix $G(x;\lambda)$ and a (spatially) constant matrix $R(\lambda)$. From here on, we will write $G(x)$ instead of $G(x;\lambda)$. By using the change of coordinates $V = G(x)^{-1}U$ we get, if $U$ is a solution of \eqref{eq:firstorderpert}, that $$G'(x)V + G(x)V' = (G(x)V)' = (\mathcal{A}_p(x) + L(x))G(x)V,$$ for every $x\in\mathbb{R}$, where $\lambda$ and $B$ have been suppressed. This implies that $$V' = G(x)^{-1}\Big(\big(\mathcal{A}_p(x) + L(x)\big)G(x) - G'(x)\Big)V.$$ Using Floquet's theorem however, we obtain $$\Phi'(x) = G'(x)e^{Rx} + G(x)Re^{Rx},$$ giving that $$G'(x) = \Phi'(x)e^{-Rx} - G(x)R = \mathcal{A}_p(x)\Phi(x)e^{-Rx} - G(x)R.$$ And so, we obtain, without suppression of $\lambda$ and $B$,
\begin{equation*}
\begin{aligned}
    V' &= G(x)^{-1}\Big(\mathcal{A}_p(x;\lambda_0)G(x) + L(x;\lambda,B)G(x) - \mathcal{A}_p(x;\lambda_0)G(x) + G(x)R(\lambda_0)\Big)V \\
    &= \Big(R(\lambda_0) + G(x)^{-1}L(x;\lambda,B)G(x)\Big)V, \text{ for every } x\in \mathbb{R}.
    \end{aligned}
\end{equation*}
Thus, the transformed system becomes
\begin{equation}
    V' = (R(\lambda_0) + S(x;\lambda,B))V,
    \label{eq:firstconstpert}
\end{equation}
where $S(x;\lambda,B) = G(x)^{-1}L(x;\lambda,B)G(x)$.
The matrices $G(x)$ and $G(x)^{-1}$ are clearly bounded since they are continuous and periodic. Hence, when $\lambda=\lambda_0$, we have $S(x;\lambda_0,B)\to 0$ as $\lvert x \rvert \to \infty$. Thus, the transformed system at infinity, i.e., the transformed version of \eqref{eq:systematinf}, can now be expressed as
\begin{equation}
    V' = R(\lambda)V.
    \label{constsysatinf}
\end{equation}

\begin{lemma}\label{lem:specmap}
The eigenvalues $\mu(\lambda)$ of the monodromy matrix $M(\lambda)$ of \eqref{eq:systematinf} and $\omega(\lambda)$ of the coefficient matrix $R(\lambda)$, for the transformed system at infinity \eqref{constsysatinf}, satisfy the relation

$$e^{p\omega(\lambda)} = \mu(\lambda).$$

\end{lemma}

\begin{proof}
From the proof of Floquet's theorem we have that the monodromy matrix satisfies $M= e^{Rp}$, and so, the result follow by the spectral mapping theorem.
\end{proof}

Note therefore that Assumption 1.2, i.e., that $1\notin \sigma(M(\lambda_0))$, is equivalent to the assumption that $2\pi i k/p \notin \sigma(R(\lambda_0))$ for $k\in\mathbb{Z}$.

\begin{lemma}\label{lem:conjeig}
If $\mu(\lambda)=e^{p\omega(\lambda)}$ is an eigenvalue of the monodromy matrix $M(\lambda)$ of \eqref{eq:systematinf}, then $\hat{\mu}(\lambda) = e^{-p\overline{\omega(\lambda)}}$ is also an eigenvalue of $M(\lambda)$. In particular, $\lvert \Hat{\mu}(\lambda) \rvert = \dfrac{1}{\lvert \mu(\lambda) \rvert}$.
\end{lemma}
\begin{proof}
This follows immediately from Theorem 3 in \cite{denk}, after setting $\omega = i\nu$.
\end{proof}
Since $M$ is a real matrix, it follows that non-real eigenvalues and corresponding eigenvectors come in complex conjugate pairs. This, together with the fact that $\det M=1$ (see Theorem 1 in \cite{denk}) and Lemma \ref{lem:conjeig} implies that the multiplicities of $\mu$ and $\hat{\mu}$ are the same. In particular, this means that the number, $2m$, of purely imaginary eigenvalues of $R(\lambda)$, i.e., the number of eigenvalues of the monodromy matrix with modulus 1, is an even number. 

\bigskip

Let $$\alpha_{\min}(\lambda)=\min_{\substack{\omega\in\sigma(R(\lambda)) \\ \real(\omega)\neq 0}}\lvert \real(\omega)\rvert >0.$$

\bigskip

Let $X^u$, $X^s$ and $X^c$ be the span of eigenfunctions corresponding to the real part of the eigenvalues being positive, negative and $0$ respectively. Further let $P^u$, $P^s$ and $P^c$ be the spectral projections onto $X^u$, $X^s$ and $X^c$.

\subsection{Exponential dichotomies}

Exponential dichotomies is a tool, originally introduced by Oskar Perron in \cite{perron}, used to investigate the stability properties and asymptotic behaviour of non-autonomous differential equations \cite{coppel}. In this paper, we use it in order to prove exponential decay of eigenfunctions, which is needed for our main result.
\begin{definition}
Let $J$ be an unbounded interval on $\mathbb{R}$. An ODE system $U' = C(x)U$ is said to possess an exponential dichotomy on $J$ if there exist constants $K>0$, $\kappa^s<0<\kappa^u$ and a family of projections $P(x_0)$, with $x_0\in J$, such that:
\begin{itemize}
    \item For any $x\in \mathbb{R}$ and $U\in\mathbb{R}^{N},$ there exists a unique solution $\Phi^s(x,x_0)U$ of the system defined for $x\geq x_0$, $x,x_0\in J$ such that 
    $$\Phi^s(x_0,x_0)U = P(x_0)U \text{ and } \lVert \Phi^s(x,x_0)U \rVert \leq Ke^{\kappa^s(x-x_0)}\lVert U \rVert.$$
    \item For any $x\in \mathbb{R}$ and $U\in\mathbb{R}^{N},$ there exists a unique solution $\Phi^u(x,x_0)U$ of the system defined for $x\leq x_0$, $x,x_0\in J$ such that 
    $$\Phi^u(x_0,x_0)U = \big(I - P(x_0)\big)U \text{ and } \lVert \Phi^u(x,x_0)U \rVert \leq Ke^{\kappa^u(x-x_0)}\lVert U \rVert.$$
    \item The solutions $\Phi^s(x,x_0)U$ and $\Phi^u(x,x_0)U$ satisfy 
    \begin{equation*}
        \begin{aligned}
            & \Phi^s(x,x_0)U \in \text{Ran}P(x) \text{ for all } x\geq x_0, \quad x,x_0 \in J \\
        & \Phi^u(x,x_0)U \in \text{ker}P(x) \text{ for all } x\leq x_0, \quad x,x_0 \in J.
        \end{aligned}
    \end{equation*}
\end{itemize}
\end{definition}
$\Phi^s$ and $\Phi^u$ are often called \textit{evolution operators} (defined for $x\geq x_0$ and $x\leq x_0$ respectively).

In order for our system \eqref{constsysatinf} to possess an exponential dichotomy, we must introduce a shift, $\pm\eta$, such that none of the eigenvalues of the matrix $R(\lambda_0) \pm\eta I$ are purely imaginary.

\begin{lemma} \label{lem:expdichsysinf}
Suppose that $\eta\in (0,\alpha_{\min}).$ Then the systems $V' = (R(\lambda_0) + \eta I)V$ and $V' = (R(\lambda_0) - \eta I)V$ each possess exponential dichotomies on $\mathbb{R}$, with $\kappa^s = -\alpha_{\min} + \eta$, $\kappa^u = \eta$ and $\kappa^s = -\eta,$ $\kappa^u = \alpha_{\min} - \eta$ respectively.
\end{lemma}

The proof follows directly from Lemma 3.1 in \cite{alexia}. Note in particular that since the shifted systems both possess exponential dichotomies on $\mathbb{R}$, it must follow that they also have it on $\mathbb{R}_+$ and $\mathbb{R}_-$.

\bigskip

The following theorem will be applied to a shifted version of the the full system (\ref{eq:firstconstpert}).
\begin{theorem}[Roughness theorem] \label{thm:rough} 
\begin{enumerate}[(i)]

    \item If $U' = C(x)U$ possesses an exponential dichotomy on $\mathbb{R}_+$ with rates $\kappa^s<0<\kappa^u$ and constant $K>0$ as in the definition, and if for some $T>0,$ $\lvert D(x) \rvert < \delta$ for all $x\geq T,$ where $\delta \in (0,\min(-\kappa^s,\kappa^u)/2K),$ then the perturbed system $U' = (C(x) + D(x))U$ also possesses an exponential dichotomy on $\mathbb{R}_+$ with rates $\Tilde{\kappa}^s = \kappa^s + 2K\delta < 0,$ $\Tilde{\kappa}^u = \kappa^u - 2K\delta >0$ and some constant $\Tilde{K}>0$.
    \item If $U' = C(x)U$ possesses an exponential dichotomy on $\mathbb{R}_-$ with rates $\kappa^s<0<\kappa^u$ and constant $K>0$ as in the definition, and if for some $T>0,$ $\lvert D(x) \rvert < \delta$ for all $x\leq -T,$ where $\delta \in (0,\min(-\kappa^s,\kappa^u)/2K),$ then the perturbed system $U' = (C(x) + D(x))U$ also possesses an exponential dichotomy on $\mathbb{R}_-$ with rates $\Tilde{\kappa}^s = \kappa^s + 2K\delta < 0,$ $\Tilde{\kappa}^u = \kappa^u - 2K\delta >0$ and some constant $\Tilde{K}>0$.
\end{enumerate}
\end{theorem}
For a proof, see \cite{alexia}.
\begin{lemma}\label{lem:fullexpdich}
Let $\eta \in (0, \alpha_{\min})$ and let $\varepsilon > 0$ be arbitrary. Then there exists a $\delta > 0$ such that if $$\lvert \lambda - \lambda_0\rvert + \sup_{x\in\mathbb{R}}\lvert B(x) \rvert < \delta,$$ then the systems 
\begin{equation}
    \mathcal{V}_{\pm}' = (R(\lambda_0) \pm \eta I + S(x;\lambda,B))\mathcal{V}_{\pm},
    \label{eq:dichfullsys}
\end{equation}
possess exponential dichotomies on $\mathbb{R}_+$ and $\mathbb{R}_-,$ respectively.

\begin{enumerate}[(i)]
    \item For the case of $\mathcal{V}_+$, the system has an exponential dichotomy on $\mathbb{R}_+$ with rates $\kappa^s = -\alpha_{\min} + \eta + \varepsilon,$ $\kappa^u = \eta - \varepsilon$.
    \item For the case of $\mathcal{V}_-$, the system has an exponential dichotomy on $\mathbb{R}_-$ with rates $\kappa^s = -\eta + \varepsilon,$ $\kappa^u = \alpha_{\min} - \eta - \varepsilon$.
\end{enumerate}
\end{lemma}
 For $\mathcal{V}_+$ on $\mathbb{R}_+$, we denote the projections by $P^s(\cdot;\lambda,B),$ and we let $P^{cu}(\cdot;\lambda,B) = I - P^s(\cdot;\lambda,B)$. We denote the corresponding evolution  operators on $\mathbb{R}_+$ by $\Psi^s(x,x_0;\lambda,B)$ and $\Psi^{cu}(x,x_0;\lambda,B)$, respectively.

 For $\mathcal{V}_-$ on $\mathbb{R}_-$, we denote the projections by $P^{cs}(\cdot;\lambda,B),$ and we let $P^{u}(\cdot;\lambda,B) = I - P^{cs}(\cdot;\lambda,B)$. We denote the corresponding evolution  operators on $\mathbb{R}_-$ by $\Psi^{cs}(x,x_0;\lambda,B)$ and $\Psi^{u}(x,x_0;\lambda,B)$, respectively.
\begin{proof}
The result follows directly from Lemma \ref{lem:expdichsysinf} and Lemma \ref{thm:rough}.
\end{proof}

Using the implicit function theorem, one can prove that the solutions of \eqref{eq:firstconstpert} are smooth in $(\lambda,B)$.
 
 \begin{lemma}\label{lem:smoothness}
 The projections $P^s(\cdot;\lambda,B)$, $P^{cu}(\cdot;\lambda,B)$, $P^{cs}(\cdot;\lambda,B)$, $P^u(\cdot;\lambda,B)$ and the corresponding evolution operators $\Psi^s(\cdot,\cdot;\lambda,B)$, $\Psi^{cu}(\cdot,\cdot;\lambda,B)$, $\Psi^{cs}(\cdot,\cdot;\lambda,B)$ and $\Psi^u(\cdot,\cdot;\lambda,B)$ depend smoothly on the parameters $\lambda$ and $B$ in a neighbourhood of $(\lambda,B)\in\mathbb{R}\times X_{\beta}.$
 \end{lemma} 
 For a proof, see \cite{alexia}.
 
 \bigskip

 Using Lemma \ref{lem:fullexpdich}, we define the evolution operators for \eqref{eq:dichfullsys} in the following way:
 
 \bigskip
 
 On $\mathbb{R}_+$ we will use the evolution operators $\Phi^s$ and $\Phi^{cu}$ defined by 
 
 \begin{equation*}
     \begin{aligned}
         \Phi^s(x,x_0;\lambda,B) &= e^{-\eta(x-x_0)}\Psi^s(x,x_0;\lambda,B), \\
         \Phi^{cu}(x,x_0;\lambda,B) &= e^{-\eta(x-x_0)}\Psi^{cu}(x,x_0;\lambda,B).
     \end{aligned}
 \end{equation*}
 
  On $\mathbb{R}_-$ we will use the evolution operators $\Phi^{cs}$ and $\Phi^{u}$ defined by 
 
 \begin{equation*}
     \begin{aligned}
         \Phi^{cs}(x,x_0;\lambda,B) &= e^{\eta(x-x_0)}\Psi^{cs}(x,x_0;\lambda,B), \\
         \Phi^{u}(x,x_0;\lambda,B) &= e^{\eta(x-x_0)}\Psi^{u}(x,x_0;\lambda,B).
     \end{aligned}
 \end{equation*}
 
 \bigskip

\section{Exponential decay of eigenfunctions} 
    The aim of this section is to prove the exponential decay of eigenfunctions, using the same ideas as in \cite{alexia}, with minor differences, to account for our slightly different scenario. In preparation, we would like to remind the reader to keep in mind the previously made Assumptions \ref{assume:1} and \ref{assume:2}.
\begin{lemma}\label{lem:vweird}
Let $V$ be a solution of (\ref{eq:firstconstpert}).

\begin{enumerate}[(i)]
    \item If $V$ is bounded on $\mathbb{R}_+$, then for every $T\geq 0$, there exists a $V_0^s\in X^s$ and a $V_0^c \in X^c$ such that for all $x\geq T$
    \begin{equation*}
        \begin{aligned}
            V(x) = e^{R(x-T)}V_0^s &+ e^{Rx}V_0^c + \int_T^x e^{R(x-\xi)}P^s(\xi;\lambda,B)S(\xi;\lambda_0,B)V(\xi) d\xi \\
            &- \int_x^{\infty} e^{R(x-\xi)}P^{cu}(\xi;\lambda,B)S(\xi;\lambda_0,B)V(\xi) d\xi,
        \end{aligned}
    \end{equation*}
    where $X^c = X^c(\lambda),$ $X^s = X^s(\lambda)$ are the span of the eigenvectors of $R=R(\lambda)$ corresponding to the purely imaginary eigenvalues and eigenvalues with real negative part of $R(\lambda)$, respectively.
    \item If $V$ bounded on $\mathbb{R}_-$, then for every $T\geq 0,$ there exists a $U_0^u\in X^u$ and $U_0^c\in X^c$ such that for all $x\leq -T$
        \begin{equation*}
        \begin{aligned}
            V(x) = e^{R(x+T)}U_0^u &+ e^{Rx}V_0^c - \int_x^{-T} e^{R(x-\xi)}P^u(\xi;\lambda,B)S(\xi;\lambda_0,B)V(\xi) d\xi \\
            &+ \int_{-\infty}^x e^{R(x-\xi)}P^{cu}(\xi;\lambda,B)S(\xi;\lambda_0,B)V(\xi) d\xi,
        \end{aligned}
    \end{equation*}
\end{enumerate}
where $X^c = X^c(\lambda),$ $X^u = X^u(\lambda)$ are the span of the eigenvectors of $R=R(\lambda)$ corresponding to the purely imaginary eigenvalues and eigenvalues with real positive part of $R(\lambda)$, respectively.
\end{lemma}
\begin{proof}

We begin with the full ODE and project using $P^s$, $P^c$ and $P^u$ as in \cite{alexia}:

$$P^i(x;\lambda,B)V'(x) = RP^i(x;\lambda,B)V(x) + P^i(x;\lambda,B)S(x;\lambda_0,B)V(x)$$
for $i \in \{s,c,u\}$.

The variation of parameters formula yields

\begin{equation}
    P^i(x;\lambda,B)V(x) = e^{R
    (x-x_0)}P^i(x_0;\lambda,B)V(x_0) + \int_{x_0}^x e^{R(x-\xi)}P^i(\xi;\lambda,B)S(\xi;\lambda_0,B)V(\xi)d\xi.
    \label{eq:varofconst}
\end{equation}
Since $\lvert V(x) \rvert$ is bounded as $x \to \infty$, it follows that $P^i(x;\lambda,B)V(x)$ must also all be bounded. Consider the equation for $P^u(x;\lambda,B)V(x)$ in \eqref{eq:varofconst}, and let $x_0 \to \infty$. Since $P^u(x_0;\lambda,B)V(x_0)$ is bounded, it follows that

$$P^u(x;\lambda,B)V(x) = - \int_x^{\infty}e^{R(x-\xi)}P^u(\xi;\lambda,B)S(\xi;\lambda_0,B)V(\xi)d\xi.$$

For $i=c$ we have that the integral in \eqref{eq:varofconst} converges as $x_0 \to \infty$. The argument follows: Since $\lVert e^{R(\lambda)x_0}P^c(x_0;\lambda,B)\rVert$ is bounded for $x_0 \in \mathbb{R}$, and since

$$S(\xi;\lambda_0,B)V(\xi) = G(\xi)^{-1}L(\xi;\lambda_0,B)G(\xi)V(\xi) = G(\xi)^{-1}L(\xi;\lambda_0,B)U(\xi),$$

we get that $S(\xi;\lambda_0,B)V(\xi) = G(\xi)^{-1}\begin{bmatrix} 0 \\
(A(\xi)-A_p(\xi) + B(\xi))U_1
\end{bmatrix}.$
This implies that
\begin{equation*}
\begin{aligned}
    \lvert S(\xi;\lambda_0,B)V(\xi) \rvert &= \lvert G(\xi)^{-1}L(\xi;\lambda_0,B)G(\xi)V(\xi) \rvert \\
    &\leq \underbrace{\lvert G(\xi)^{-1}}_{\leq C_1} \rvert \cdot \lvert (A(\xi) - A_p(\xi) + B(\xi))U_1(\xi) \rvert \\
    &\leq C_1 (1+\xi)^{\beta}(1+\xi)^{-\beta}  \cdot \lvert (A(\xi) - A_p(\xi) + B(\xi))U_1(\xi) \rvert\\
    &\leq C_1(1+\xi)^{-\beta}\lvert U(\xi) \rvert \sup_{\xi\in\mathbb{R}}\Big((1+\xi)^{\beta}\lvert A(\xi) - A_p(\xi) + B(\xi)\rvert\Big)\\
    &=C_1(1+\xi)^{-\beta}\lVert A-A_p + B\rVert_{X^{\beta}}\lvert G(\xi)V(\xi)\rvert \\
    &\leq C_1(1+\xi)^{-\beta}\lVert A-A_p + B\rVert_{X^{\beta}}\underbrace{\lvert G(\xi) \rvert}_{\leq C_2} \cdot \lvert V(\xi)\rvert \\
    &\leq C(1+\xi)^{-\beta}\lVert A-A_p + B\rVert_{X^{\beta}}\lvert V(\xi)\rvert.
    \end{aligned}
\end{equation*}
Here $C=C_1C_2$, where $C_1 = \adjustlimits\sup_{\xi\in\mathbb{R}}\lvert G(\xi)^{-1}\rvert$ and $C_2 = \adjustlimits\sup_{\xi\in\mathbb{R}}\lvert G(\xi)\rvert$. This is because $G
(x)$ is invertible for all $x\in \mathbb{R}$ and is continuous and periodic, so $G(\xi)$ and $G(\xi)^{-1}$ attain a global maximum. 
Thus
\begin{equation}
\begin{aligned}
    \int_x^{\infty}\lvert S(\xi;\lambda_0,B)&V(\xi) \rvert d\xi \\
    &\leq C(\lVert A-A_p \rVert_{X^{\beta}} + \lVert B \rVert_{X^{\beta}})\sup_{\xi \geq x} \lvert V(\xi) \rvert \int_x^{\infty}(1+\xi)^{-\beta} d\xi \\
    &\leq C\dfrac{1}{\beta-1}(\lVert A-A_p \rVert_{X^{\beta}} + \lVert B \rVert_{X^{\beta}})\lVert V \rVert_{\infty}\dfrac{1}{(1+x)^{\beta-1}}.
\end{aligned}
    \label{eq:integralest1}
\end{equation}

For the first term of $P^c(x;\lambda,B)V(x)$, we have that the limit

$$\lim_{x_0 \to \infty} e^{-Rx_0}P^c(x_0;\lambda,B)V(x_0)  = V_0^c$$
exists. This follows from the fact that the left-hand side of the equation does not depend on $x_0$ and the integral is convergent. Hence

$$ P^c(x;\lambda,B)V(x) = e^{Rx}V_0^c - \int_x^{\infty} e^{R(x-\xi)}P^c(\xi;\lambda,B)S(\xi;\lambda_0,B)V(\xi) d\xi.$$

For the last case, $P^s(x;\lambda,B)V(x)$, we pick $x_0 = T \geq 0$, which for $x\geq T$ gives us that

\begin{equation*}
    \begin{aligned}
    V(x) &= P^s(x;\lambda,B)V(x) + P^c(x;\lambda,B)V(x) + P^u(x;\lambda,B)V(x) \\
    &= e^{R(x-T)}V_0^s + e^{Rx}V_0^c + \int_T^xe^{R(x-\xi)}P^s(\xi;\lambda,B)S(\xi;\lambda_0,B)V(\xi)d\xi \\
    &- \int_x^{\infty}e^{R(x-\xi)}P^{cu}(\xi;\lambda,B)S(\xi;\lambda_0,B)V(\xi)d\xi,
    \end{aligned}
\end{equation*}
where $V_0^s = P^s(T;\lambda,B)V(T).$

The proof of $(ii)$ follows similarly.
\end{proof}

With the help of Lemma \ref{lem:vweird} above, we can actually prove that any eigenfunction of the perturbed operator, $\mathcal{L}_B$, decays exponentially.

\begin{lemma}\label{lem:expdecay}
Let $\lambda$ be an eigenvalue of the perturbed operator $\mathcal{L}_B$, and assume that Assumption \ref{assume:1} and Assumption \ref{assume:2} hold. Further, let $\mathbf{u}\in L^2(\mathbb{R};\mathbb{R}^n)$ be the corresponding eigenfunction, and $\hat{\kappa} \in (0, \alpha_{\min}(\lambda))$. Denote by $V$ the solution of (\ref{eq:firstconstpert}) corresponding to $\mathbf{u}$. Then, there exists a positive constant $K$ such that 
$$\lvert V(x) \rvert \leq Ke^{-\hat{\kappa}\lvert x\rvert} \text{ for all } x\in\mathbb{R}.$$
\end{lemma}

Given Lemma \ref{lem:vweird}, the proof of Lemma \ref{lem:expdecay} is completely similar to that of Lemma 4.2 in \cite{alexia}.

\bigskip

\section{Lyapunov-Schmidt reduction}
In this section we provide some further results, until we can finally prove the main theorem.

Let $\mathbf{u}_*$ be the eigenfunction to the unperturbed eigenvalue equation $\mathcal{L}\mathbf{u}_* = \lambda_0\mathbf{u}_*$. We shall assume that $\mathbf{u}_*$ is normalized, i.e., that $\langle \mathbf{u}_*,\mathbf{u}_* \rangle_{L^2} = 1$. We further denote $U_* = (\mathbf{u}_*,\mathbf{u}_*')^T$ and $V_* = G(x)^{-1}U_*$.

We define the stable and unstable subspaces $E_+^s$ and $E_-^u$ respectively. They consist of initial conditions for which the unperturbed system decays exponentially for increasing and decreasing values of $x$ respectively, and are defined as

\begin{equation}
    \begin{aligned}
    E_+^s &= \{V\in \mathbb{R}^{2n};\hspace{2mm} P^s(T;\lambda_0,0)V = V \}, \\
    E_-^u &= \{V\in \mathbb{R}^{2n};\hspace{2mm} P^u(-T;\lambda_0,0)V = V \}. \\
    \end{aligned}
    \label{eq:unstsub}
\end{equation}

We further define the mapping $\iota:E_+^s \times E_-^u \times \mathbb{R} \times X_{\beta} \to \mathbb{R}^{2n}$ by

\begin{equation}
    \iota(V_0^s,V_0^u;\lambda,B) = \Phi(0,T;\lambda,B)P^s(T;\lambda,B)V_0^s - \Phi(0,-T;\lambda,B)P^u(-T;\lambda,B)V_0^u.
    \label{eq:iotadef}
\end{equation}
Here, $\Phi(x,x_0)$ are the standard state-transition matrices defined through the fundamental matrix solution by $\Phi(x,x_0)=\Phi(x)\Phi(x_0)^{-1}$.
\begin{lemma}
Let $\lambda$ be such that $e^{p\lambda} \notin \sigma(M)$, where $M$ is the monodromy matrix of the unperturbed system at infinity. Further, let $\delta>0$ be such that Lemma \ref{lem:fullexpdich} holds. Then $\lambda$ is an eigenvalue of $\mathcal{L}_B$ if and only if there exists $V_0^s \in E_+^s$ and $V_0^u \in E_-^u$ with $(V_0^s,V_0^u) \neq 0$ such that
\begin{equation}
    \iota(V_0^s,V_0^u;\lambda,B) = 0.
    \label{eq:iota0}
\end{equation}
\end{lemma}
\begin{proof}
Let us first assume that \eqref{eq:iota0} is fullfilled. Then it follows that
$$\Phi(0,T;\lambda,B)P^s(T;\lambda,B)V_0^s = \Phi(0,-T;\lambda,B)P^u(-T;\lambda,B)V_0^u,$$
or that solutions to the system with initial value 
$$V(0) = \Phi(0,T;\lambda,B)P^s(T;\lambda,B)V_0^s = \Phi(0,-T;\lambda,B)P^u(-T;\lambda,B)V_0^u$$
decays exponentially as $\lvert x\rvert \to\infty$.
Then $\lambda$ has to be an eigenvalue of the perturbed operator $\mathcal{L}_B$, and the corresponding eigenfunction is the first component of $U(x)=G(x)V(x).$

On the other hand, if we assume that $\lambda$ is an eigenvalue of $\mathcal{L}_B$, then by Lemma \ref{lem:expdecay}, there is a solution $U(x)$ of the system \eqref{eq:1stordperteq}, and hence a solution $V(x) = G(x)^{-1}U(x)$ of the system \eqref{eq:firstconstpert}, which decays exponentially as $\lvert x \rvert \to \infty$.

Pick 

\begin{equation*}
\begin{aligned}
&V_0^s = P^s(T;\lambda_0,0)V(T) \in E_+^s, \\
&V_0^u = P^u(-T;\lambda_0,0)V(-T)\in E_-^u.
\end{aligned}
\end{equation*}

Since we have that $P^s(T;\lambda,B)P^s(T;\lambda_0,0)=P^s(T;\lambda,B)$ from the proof of the roughness theorem, together with definition of exponential dichotomies, we get that
$$\Phi(0,T;\lambda,B)P^s(T;\lambda,B)V_0^s = V(0) = \Phi(0,-T;\lambda,B)P^u(-T;\lambda,B)V_0^u,$$
at which point we are done.
\end{proof}

Note that, for any $(V_0^s,V_0^u) \in E_+^s \times E_-^u$, we have 

$$\iota(V_0^s,V_0^u;\lambda_0,0) = \Phi(0,T;\lambda_0,0)V_0^s - \Phi(0,-T;\lambda_0,0)V_0^u. $$ Hence, $\text{Ran }\iota(\cdot,\cdot,\lambda_0,0) = \Phi(0,T;\lambda_0,0)E_+^s + \Phi(0,-T;\lambda_0,0)E_-^u.$

\begin{lemma}\label{lem:codimthing}
We have that 

$$\codim(\Phi(0,T;\lambda_0,0)E_+^s + \Phi(0,-T;\lambda_0,0)E_-^u) = 2m+1.$$
\end{lemma}

\begin{proof}
The number of eigenvalues of $R(\lambda_0)$ with non-zero real part is $2(n-m)$ since $n-m=\dim X^s = \dim X^u$. As $\dim X^s = \dim E_+^s$ and $\dim X^u = \dim E_-^u$, we obtain

\begin{equation*}
    \begin{aligned}
    \dim &(\Phi(0,T;\lambda_0,0)E_+^s + \Phi(0,-T;\lambda_0,0)E_-^u) \\
    &= \dim(E_+^s) + \dim(E_-^u) -  \dim(\Phi(0,T;\lambda_0,0)E_+^s \cap \Phi(0,-T;\lambda_0,0)E_-^u) \\
    &= 2(n-m)-1,
    \end{aligned}
\end{equation*}
and so,
$$\codim(\Phi(0,T;\lambda_0,0)E_+^s + \Phi(0,-T;\lambda_0,0)E_-^u) = 2n - 2n + 2m + 1 = 2m+1. $$

\end{proof}

Let Q be a projection in $\mathbb{R}^{2n}$ onto $\text{Ran }\iota (\cdot,\cdot;\lambda_0,0) = \Phi(0,T;\lambda_0,0)E_+^s + \Phi(0,-T;\lambda_0,0)E_-^u.$ Then we can write \eqref{eq:iota0} as 

\begin{equation}
    \begin{aligned}
    &Q\iota(V_0^s, V_0^u, \lambda,B) = 0, \\
    (I-&Q)\iota(V_0^s, V_0^u, \lambda,B) = 0.
    \end{aligned}
    \label{eq:lyapsys}
\end{equation}

It follows from Lemma \ref{lem:codimthing} that 

$$\dim(\ker Q) = \codim(\Phi(0,T;\lambda_0,0)E_+^s + \Phi(0,-T;\lambda_0,0)E_-^u) = 2m+1.$$

We use Lyaounov-Schmidt reduction to solve \eqref{eq:iota0}, i.e., we solve the first equation of \eqref{eq:lyapsys} for $V_0^s$ and $V_0^u$ in terms of $(\lambda,B)$ using the implicit function theorem, and then substitute the solution $V_0^s = V_0^s(\lambda,B)$ and $V_0^u = V_0^u(\lambda,B)$ into the second equation of \eqref{eq:lyapsys} and solve this equation, again using the implicit function theorem, reducing everything to functions of $B$. In order to get a unique solution we need to add another condition which fixes this solution amongst infinitely many, in a one-dimensional subspace of solutions.

\begin{lemma}\label{lem:hyperplane}
Let $D$ be a subspace of $E_+^s \times E_-^u$ such that $\text{span}\{(V_*(T),V_*(-T))\} + D = E_+^s \times E_-^u$ and $D \cap \text{span}\{(V_*(T),V_*(-T))\} = \{0\}.$  Then for $(\lambda,B)$ close to $(\lambda_0,0)$, the first equation of \eqref{eq:lyapsys}, has a unique solution
$$(V_0^s,V_0^u) = (V_0^s(\lambda,B),V_0^u(\lambda,B)),$$
such that $(V_0^s,V_0^u) - (V_*(T),V_*(-T))\in D$.
\end{lemma}
\begin{proof}
The proof is completely similar to that of Lemma 5.3 in \cite{alexia}, and will thus be omitted.
\end{proof}

By the integral formula derived in the proof of the roughness theorem \cite[p.~30]{coppel}, we obtain 

\begin{equation*}
\begin{aligned}
    \Phi(0,T;\lambda,B)&P^s(T;\lambda,B) = \Phi(0,T;\lambda_0,0)P^s(T;\lambda_0,0) \\
    & - \int_0^T \Phi(0,\xi;\lambda_0,0)P^s(\xi;\lambda_0,0)\mathcal{N}(\xi;\lambda,B)\Phi(\xi,T;\lambda,B)P^s(T;\lambda,B)d\xi \\
    &- \int_0^{\infty}\Phi(0,\xi)P^u(\xi;\lambda_0,0)\mathcal{N}(\xi;\lambda,B)\Phi(\xi,T;\lambda,B)P^s(T;\lambda,B)d\xi,
\end{aligned}
\end{equation*}

where $\mathcal{N}(\xi;\lambda,B) = G(\xi)^{-1}N(\xi;\lambda,B)G(\xi)$, and

$$N(\xi;\lambda,B) = \begin{bmatrix} 0 & 0 \\
B(\xi) - (\lambda-\lambda_0)I & 0
\end{bmatrix}.$$

Similarly, we have

\begin{equation*}
\begin{aligned}
    \Phi(0,-T;\lambda,B)&P^u(-T;\lambda,B) = \Phi(0,-T;\lambda_0,0)P^u(-T;\lambda_0,0) \\
    & + \int_{-T}^0 \Phi(0,\xi)P^u(\xi;\lambda_0,0)\mathcal{N}(\xi;\lambda,B)\Phi(\xi,-T;\lambda,B)P^u(-T;\lambda,B)d\xi \\
    &+ \int_{-\infty}^0\Phi(0,\xi)P^s(\xi,\lambda_0,0)\mathcal{N}(\xi;\lambda,B)\Phi(\xi,-T;\lambda,B)P^u(-T;\lambda,B)d\xi.
\end{aligned}
\end{equation*}

Combining these two with \eqref{eq:iotadef} yields

\begin{equation*}
\begin{aligned}
    \iota(V_0^s&,V_0^u;\lambda,B) = \Phi(0,T;\lambda_0,0)V_0^s - \Phi(0,-T;\lambda_0,0)V_0^u \\  
    & - \int_0^T \Phi(0,\xi)P^s(\xi;\lambda_0,0)\mathcal{N}(\xi;\lambda,B)\Phi(\xi,T;\lambda,B)P^s(T;\lambda,B)V_0^sd\xi \\
    &- \int_0^{\infty}\Phi(0,\xi)P^u(\xi;\lambda_0,0)\mathcal{N}(\xi;\lambda,B)\Phi(\xi,T;\lambda,B)P^s(T;\lambda,B)V_0^sd\xi \\
    & - \int_{-T}^0 \Phi(0,\xi)P^u(\xi;\lambda_0,0)\mathcal{N}(\xi;\lambda,B)\Phi(\xi,-T;\lambda,B)P^u(-T;\lambda,B)V_0^ud\xi \\
    &- \int_{-\infty}^0\Phi(0,\xi)P^s(\xi;\lambda_0,0)\mathcal{N}(\xi;\lambda,B)\Phi(\xi,-T;\lambda,B)P^u(-T;\lambda,B)V_0^ud\xi.
\end{aligned}
\end{equation*}

To solve the second equation of \eqref{eq:lyapsys}, we define the map $F:\mathbb{R}\times X_{\beta} \to \ker Q$ by

\begin{equation}
    F(\lambda,B) = \iota(V_0^s(\lambda,B),V_0^u(\lambda,B);\lambda,B) = (I-Q)\iota(V_0^s(\lambda,B),V_0^u(\lambda,B);\lambda,B).
    \label{eq:Fdef}
\end{equation}

Solving $F(\lambda,B) = 0$ is then equivalent to solving \eqref{eq:lyapsys}.

\bigskip

In order to do that, we define the adjoint equation for $\lambda=\lambda_0$ and $B=0$

\begin{equation}
W' = -(R(\lambda_0) + S(x;\lambda_0,0))^*W.
\label{eq:adjointeq}
\end{equation}

 Furthermore, $\Psi(x) = (\Phi(x)^*)^{-1}$ is the fundamental matrix for the adjoint equation \eqref{eq:adjointeq}, which can be verified by direct calculations. For details, see \cite[p.~70]{codd}. This system clearly possesses exponential dichotomies on $\mathbb{R}_+$ and $\mathbb{R}_-$, by the same reasons as \eqref{eq:firstconstpert}, i.e., using that the corresponding system at infinity possesses exponential dichotomies and then applying the roughness theorem. The exponential dichotomies are denoted by $\Psi^s(x,x_0)$, $\Psi^{cu}(x,x_0)$ and $\Psi^u(x,x_0)$, $\Psi^{cs}(x,x_0)$ respectively. They are related to the exponential dichotomies of the unperturbed system in \eqref{eq:firstconstpert}, with $\lambda=\lambda_0$ and $B=0$,  in the following way
 
 \begin{equation*}
     \begin{aligned}
     \Psi^s(x,x_0) = \Phi^{cu}(x_0,x)^*, \quad  \Psi^{cu}(x,x_0) = \Phi^{s}(x_0,x)^*, \\
     \Psi^{cs}(x,x_0) = \Phi^{u}(x_0,x)^*, \quad  \Psi^u(x,x_0) = \Phi^{cs}(x_0,x)^*.
     \end{aligned}
 \end{equation*}

The adjoint equation for \eqref{eq:firstorderpert} is 

$$Z' = -(\mathcal{A}_p(x,\lambda_0) + L(x;\lambda_0,0))^*Z.$$

A solution of this equation is $Z_*=U_*^{\perp} = (-\mathbf{u}_*',\mathbf{u}_*)^T$, where $U_*=(\mathbf{u}_*,\mathbf{u}_*')$ is the solution of the original equation, \eqref{eq:1stordperteq}. In fact, by direct calculations, it is clear that $V_*^{\perp}=W_*=G(x)^*Z_*$ solves \eqref{eq:adjointeq}.

Further, since 
$$\dfrac{d}{dx}\langle V,W \rangle = 0,$$
it follows that 
$$\langle V,W \rangle = C,$$
for some constant $C$. In particular, for $V=V_*$ and $W=V_*^{\perp}$ at $x=0$, we have that 
$$\langle V_*(0),V_*^{\perp}(0)\rangle = \langle U_*(0), U_*^{\perp}(0) \rangle = 0.$$
It follows that $\langle V_*,V_*^{\perp} \rangle = 0$ for all $x \in \mathbb{R}$. Clearly, since $V_*^{\perp}=G(x)^*U_*^{\perp}$, $V_*^{\perp}$ decays exponentially as $\lvert x\rvert \to \infty.$

For $V^s\in E_+^s$ and $V^u\in E_-^u$ it holds that
\begin{equation*}
    \begin{aligned}
    \dfrac{d}{dx}\langle V_*^{\perp}(x), \Phi^s(x,T)V^s \rangle &= \dfrac{d}{dx}\langle \Psi^s(x,T)V_*^{\perp}(T), \Phi^s(x,T)V^s \rangle  \rangle = 0, \\
    \dfrac{d}{dx}\langle V_*^{\perp}(x), \Phi^u(x,-T)V^u \rangle &= \dfrac{d}{dx}\langle \Psi^u(x,T)V_*^{\perp}(T), \Phi^u(x,-T)V^u \rangle  \rangle = 0,
    \end{aligned}
\end{equation*}
which follows by definition through direct calculations.
Thus, $\langle V_*^{\perp}, \Phi^s(x,T)V^s \rangle$ and $\langle V_*^{\perp}, \Phi^u(x,-T)V^u \rangle$ are both constant. Moreover, since $\Phi^s(x,T)V^s \to 0$ as $x\to \infty$, $\Phi^u(x,-T)V^u \to 0$ as $x\to -\infty$ and $V_*^{\perp}$ is bounded on $\mathbb{R}$, we obtain
$$\langle V_*^{\perp}(0), \Phi
^s(0,T)V^s + \Phi^u(0,-T)V^u \rangle = 0.$$
Hence, for every $x \in \mathbb{R}$, we have
$$\langle V_*^{\perp}(x), \Phi^s(x,T)V^s + \Phi^u(x,-T)V^u \rangle = 0,$$
and for any $V\in\mathbb{R}^{2n}$
\begin{equation}
    \langle V_*^{\perp}(0), QV\rangle = \langle Q^*V_*^{\perp}(0), V\rangle = 0,
    \label{eq:projprod}
\end{equation}
from which it follows that $V_*^{\perp}(0) \in \ker Q^*$.

\bigskip

\begin{lemma}\label{lem:lambfunc}
The equation $\langle V_*^{\perp}(0), F(\lambda,B) \rangle = 0$ defines a smooth function $\lambda(B)$ in a neighbourhood of $B=0$ such that $\lambda(0) = \lambda_0$. Also, for amy $B\in X_{\beta}$

$$\lambda'(0)B = - \int_{-\infty}^{\infty} \langle \mathbf{u}_*(x),B(x)\mathbf{u}_*(x)\rangle dx.$$

\end{lemma}

\begin{proof}
Define $F_*(\lambda,B) = \langle V_*^{\perp}(0), F(\lambda,B) \rangle$. We want to solve the equation $F_*(\lambda,B)=0$. Using equation \eqref{eq:Fdef}, we get that

\begin{equation*}
    \begin{aligned}
    F_*(\lambda,B) &= -\int_0^T\langle V_*^{\perp}(0), \Phi(0,\xi;\lambda_0,0)P^s(\xi;\lambda_0,0)\mathcal{N}(\xi;\lambda,B)\Phi(\xi,T;\lambda,B)P^s(T;\lambda,B)V_0^s(\lambda,B) \rangle d\xi \\
    &-\int_{-T}^0\langle V_*^{\perp}(0), \Phi(0,\xi;\lambda_0,0)P^u(\xi;\lambda_0,0)\mathcal{N}(\xi;\lambda,B)\Phi(\xi,-T;\lambda,B)P^s(-T;\lambda,B)V_0^u(\lambda,B) \rangle d\xi \\
    &-\int_{0}^{\infty}\langle V_*^{\perp}(0), \Phi(0,\xi;\lambda_0,0)(I-P^s(\xi;\lambda_0,0))\mathcal{N}(\xi;\lambda,B)\Phi(\xi,T;\lambda,B)P^s(T;\lambda,B)V_0^s(\lambda,B) \rangle d\xi \\
    &-\int_{-\infty}^{0}\langle V_*^{\perp}(0), \Phi(0,\xi;\lambda_0,0)(I-P^u(\xi;\lambda_0,0))\mathcal{N}(\xi;\lambda,B)\Phi(\xi,-T;\lambda,B)P^u(-T;\lambda,B)V_0^u(\lambda,B) \rangle d\xi.
    \end{aligned}
\end{equation*}

By Lemma \ref{lem:smoothness}, $F_*$ is a smooth function of $\lambda$ and $B$ in a neighbourhood of $(\lambda_0,0)$, with $F_*(\lambda_0,0)=0$. Since $\mathcal{N}(\xi;\lambda_0,0)=0$ for all $\xi \in \mathbb{R}$ it follows, after differentiating at $(\lambda_0,0)$ with respect to $\lambda$, that

\begin{equation*}
    \begin{aligned}
    \dfrac{\partial F_*}{\partial\lambda}(\lambda_0,0) &= -\int_0^T\langle V_*^{\perp}(0), \Phi(0,\xi;\lambda_0,0)P^s(\xi;\lambda_0,0)\dfrac{\partial \mathcal{N}}{\partial\lambda}(\xi)\Phi(\xi,T;\lambda_0,0)P^s(T;\lambda_0,0)V_0^s(\lambda_0,0) \rangle d\xi \\
    &-\int_{-T}^0\langle V_*^{\perp}(0), \Phi(0,\xi;\lambda_0,0)P^u(\xi;\lambda_0,0)\dfrac{\partial \mathcal{N}}{\partial\lambda}(\xi)\Phi(\xi,-T;\lambda_0,0)P^s(-T;\lambda_0,0)V_0^u(\lambda_0,0) \rangle d\xi \\
    &-\int_{0}^{\infty}\langle V_*^{\perp}(0), \Phi(0,\xi;\lambda_0,0)(I-P^s(\xi;\lambda_0,0))\dfrac{\partial \mathcal{N}}{\partial\lambda}(\xi)\Phi(\xi,T;\lambda_0,0)P^s(T;\lambda_0,0)V_0^s(\lambda_0,0) \rangle d\xi \\
    &-\int_{-\infty}^{0}\langle V_*^{\perp}(0), \Phi(0,\xi;\lambda_0,0)(I-P^u(\xi;\lambda_0,0))\dfrac{\partial \mathcal{N}}{\partial\lambda}(\xi)\Phi(\xi,-T;\lambda_0,0)P^u(-T;\lambda_0,0)V_0^u(\lambda_0,0) \rangle d\xi,
    \end{aligned}
\end{equation*}

where

$$\dfrac{\partial\mathcal{N}}{\partial\lambda}(\xi) = G(\xi)^{-1}\begin{bmatrix} 0 & 0 \\ -I & 0 \end{bmatrix}G(\xi).$$

The first two integrals above are zero due to \eqref{eq:projprod}. Hence 

\begin{equation}
    \begin{aligned}
    \dfrac{\partial F_*}{\partial\lambda}&(\lambda_0,B) = - \int_0^{\infty}\langle V_*^{\perp}(\xi), \dfrac{\partial\mathcal{N}}{\partial\lambda}(\xi)V_*(\xi)\rangle d\xi - \int_{-\infty}^{0}\langle V_*^{\perp}(\xi), \dfrac{\partial\mathcal{N}}{\partial\lambda}(\xi)V_*(\xi)\rangle d\xi \\
    &= -\int_{-\infty}^{\infty}\langle V_*^{\perp}(\xi), \dfrac{\partial\mathcal{N}}{\partial\lambda}(\xi)V_*(\xi)\rangle d\xi = -\int_{-\infty}^{\infty}\langle G(\xi)^*U_*^{\perp}(\xi), G(\xi)^{-1}\dfrac{\partial N}{\partial\lambda}G(\xi)G(\xi)^{-1}U_*(\xi)\rangle d\xi \\
    &= -\int_{-\infty}^{\infty}\langle U_*^{\perp}(\xi), \dfrac{\partial N}{\partial\lambda}U_*(\xi)\rangle d\xi = \int_{-\infty}^{\infty} \lvert \mathbf{u}_*(\xi)\rvert^2 d\xi = 1.
    \end{aligned}
    \label{eq:lambdaderr}
\end{equation}

Differentiating using the chain rule yields 

$$\lambda'(0)B = -(\partial_{\lambda}F_*(\lambda_0,0))^{-1}(\partial_BF_*(\lambda_0,0))B.$$

By \eqref{eq:lambdaderr}, and similar arguments for $\partial_BF_*(\lambda_0,0)$, it follows that 

$$\lambda'(0)B=\partial_BF_*(\lambda_0,0)B=-\int_{-\infty}^{\infty}\langle \mathbf{u}_*(x), B(x)\mathbf{u}_*(x) \rangle,$$

which was the assertion.

\end{proof}

To prove Theorem \ref{thm:main}, we need the follwing result.
\begin{lemma}\label{lem:intcovlem}
Let $\mathbf{v}: \mathbb{R} \to \mathbb{R}^n$ be a continuous function. If for every $B\in X_{\beta}$

$$\int_{\mathbb{R}} \langle \mathbf{v}(\xi),B(\xi)\mathbf{u}_*(\xi)\rangle dx=0,$$

then $\mathbf{v}=0$, where $\mathbf{u}_*$ is the eigenfunction corresponding to the embedded eigenvalue $\lambda_0$ of the unperturbed operator $\mathcal{L}$ from \eqref{eq:opdef}.
\end{lemma}
\begin{proof}
See Lemma 5.5 in \cite{alexia}, where it is shown that the assertion follows from the fundamental lemma of calculus of variations.
\end{proof}
In order to be able to prove Theorem \ref{thm:second} however, we will need a slightly different result.
\begin{lemma}\label{lem:modcovlem}
    Assume that Assumption \ref{assume:3} holds. Let $\mathbf{v}: \mathbb{R} \to \mathbb{R}^n$ be a continuous function, and let $\mathbf{u}_*$ be the eigenfunction corresponding to the embedded eigenvalue $\lambda_0$ of the unperturbed operator $\mathcal{L}$ from \eqref{eq:opdef}. If for every $B\in Y_{\beta}$

$$\int_{\mathbb{R}} \langle \mathbf{v}(\xi),B(\xi)\mathbf{u}_*(\xi)\rangle dx=0,$$

then $v_j=0$ for all $j \neq 1$. 
\end{lemma}
\begin{proof}
Straightforward computation then yields
$$\langle \mathbf{v}(\xi),B\mathbf{u}_*(\xi)\rangle= \sum_{i\neq 1}b_{i1}(\xi)v_i(\xi)u_{*1}(\xi),$$

    We pick a $B\in Y_{\beta}$ such that for some specific $k\neq 1$, we have $b_{1k}(\xi)=b_{k1}(\xi) \neq 0$, and all other entries are equal to zero, where $b_{1k}$ smooth and compactly supported. Then
    $$\int_{\mathbb{R}} \langle \mathbf{v}(\xi),B(\xi)\mathbf{u}_*(\xi)\rangle dx=\int_{\mathbb{R}}b_{1k}(\xi) v_k(\xi)u_{*1}(\xi)d\xi =0.$$
    By the fundamental lemma of calculus of variations it follows that
    \begin{equation}
    v_k(\xi)u_{*1}(\xi)=0,
    \label{eq:zeroeq}
    \end{equation}
     and with the same reasoning as in the proof of Lemma \ref{lem:intcovlem} in \cite{alexia}, it follows that $v_k(\xi)=0$ for all $\xi\in\mathbb{R}$. Since $k\neq 1$ was picked arbitrarily, the result follows that $v_k=0$ for all $k\neq 1$.
\end{proof}

\bigskip

\begin{proof}[Proof of Theorem \ref{thm:main}]
Since $\dim(\ker Q^*) = \dim(\ker Q) = 2m+1$, and we know that $V_*^{\perp}(0) \in \ker Q^*$, we define $W_k(0)\in\mathbb{R}^{2n}$ for $k=1,...,2m$ such that $\{W_k(0); k=1,...,2m\} \cup \{V_*^{\perp}(0)\}$ is a basis for $\ker Q^*$. Let $W_k$ be the solution of the adjoint unperturbed system with initial value $W_k(0)$.
Further, define $F_k: X_{\beta} \to \mathbb{R}$ by

$$F_k(B) = \langle W_k(0), F(\lambda(B),B)\rangle, \quad k=1,...,2m,$$ with $F$ as in \eqref{eq:Fdef}. All $F_k$ are clearly smooth, since $F$ is by Lemma \ref{lem:smoothness}. If $F_k(B)=0$ for some $B\in X_{\beta}$ for all $k=1,...,2m$, then $F(\lambda(B),B)=0$ since $\{W_k(0); k=1,...,2m\}\cup \{V_*^{\perp}(0)\}$ is a basis for $\ker Q^*$. The converse clearly holds as well. 

Now, to show that there is a $2m$-dimensional manifold of perturbations $B$ defined by the equations $F_k(B)=0, \quad k=1,...,2m$, we start by defining 

\begin{equation*}
  V(x,B) = \left\{
        \begin{array}{ll}
            \Phi(x,T;\lambda(B),B)V_0^s(\lambda(B),B) \quad \text{for } x\geq 0,   \\
            \Phi(x,-T;\lambda(B),B)V_0^u(\lambda(B),B) \quad \text{for } x< 0 ,
        \end{array}.
    \right.
\end{equation*}
giving that
\begin{equation*}
    \begin{aligned}
    F_k(B) &= -\int_0^T\langle W_k(0), \Phi(0,\xi;\lambda_0,0)P^s(\xi;\lambda_0,0)\mathcal{N}(\xi,\lambda,B)V(\xi,B) \rangle d\xi \\
    &-\int_{-T}^0\langle W_k(0), \Phi(0,\xi;\lambda_0,0)P^u(\xi;\lambda_0,0)\mathcal{N}(\xi,\lambda,B)V(\xi,B) \rangle d\xi \\
    &-\int_{0}^{\infty}\langle W_k(0), \Phi(0,\xi;\lambda_0,0)(I-P^s(\xi;\lambda_0,0))\mathcal{N}(\xi,\lambda,B)V(\xi,B) \rangle d\xi \\
    &-\int_{-\infty}^{0}\langle W_k(0), \Phi(0,\xi;\lambda_0,0)(I-P^u(\xi;\lambda_0,0))\mathcal{N}(\xi,\lambda,B)V(\xi,B) \rangle d\xi.
    \end{aligned}
\end{equation*}
Since $W_k(0) = (I-Q^*)W_k(0),$ the first two terms above are 0. Hence
\begin{equation*}
    \begin{aligned}
    F_k(B) &= -\int_{-\infty}^{\infty} \langle W_k(0), \Phi(0,\xi;\lambda_0,0)\mathcal{N}(\xi;B)V(\xi,B)\rangle d\xi \\
    &= -\int_{-\infty}^{\infty} \langle W_k(\xi),\mathcal{N}(\xi;\lambda,B)V(\xi,B)\rangle d\xi \\
    &= -\int_{-\infty}^{\infty} \langle G(\xi)^*Z_k(\xi),G(\xi)^{-1}N(\xi;\lambda,B)G(\xi)G(\xi)^{-1}U(\xi,B)\rangle d\xi \\
    &= -\int_{-\infty}^{\infty} \langle Z_k(\xi),N(\xi;\lambda,B)U(\xi,B)\rangle d\xi \\
    &= -\int_{-\infty}^{\infty} \langle \mathbf{z}_k(\xi),(B-(\lambda(B)-\lambda_0)I)\mathbf{u}(\xi;B)\rangle d\xi, \\
    \end{aligned}
\end{equation*}
where $Z_k=(-\mathbf{z}_k',\mathbf{z}_k)^T$ and $\mathbf{u}(\xi,B)$ is the vector comprised of the first $n$ components of $U(\xi,B)$.

Next, we wish to prove that $F_k'(0), k=1,...,2m$ are linearly independent. Using Lemma \ref{lem:hyperplane} and Lemma \ref{lem:lambfunc}, we get that
\begin{equation}
\begin{aligned}
    F_k'(0)B &= -\int_{-\infty}^{\infty} \langle \mathbf{z}_k(\xi),(B(\xi) - \lambda'(0)B)\mathbf{u}_*(\xi)\rangle d\xi \\
    &= -\int_{-\infty}^{\infty} \langle \mathbf{z}_k, B(\xi)\mathbf{u}_*(\xi) \rangle d\xi - \int_{-\infty}^{\infty} \langle \mathbf{u}_*(\xi), B(\xi)\mathbf{u}_*(\xi) \rangle d\xi\int_{-\infty}^{\infty} \langle \mathbf{z}_k(\xi),\mathbf{u}_*(\xi) \rangle d\xi.
\end{aligned}
\label{eq:fkprim}
\end{equation}

Let $\alpha_k \in \mathbb{R}$ for all $k=1,...,2m$ be such that for every $B\in X_{\beta}$,
$$\sum_{k=1}^{2m}\alpha_kF_k'(0)B=0.$$
Then, setting 
$$\alpha_0 = \int_{-\infty}^{\infty} \langle \mathbf{z}(\xi), \mathbf{u}_*(\xi) \rangle d\xi$$
where
$$\mathbf{z}(\xi) = \sum_{k=1}^{2m}\alpha_k\mathbf{z}_k(\xi),$$
we arrive at
$$\int_{-\infty}^{\infty}\langle \mathbf{z}(\xi) + \alpha_0\mathbf{u}_*(\xi),B(\xi)\mathbf{u}_*(\xi)\rangle d\xi=0.$$
Applying Lemma \ref{lem:intcovlem} with $\mathbf{v}=\mathbf{z}+ \alpha_0\mathbf{u}_*$ yields
$$\alpha_0\mathbf{u}_*(\xi) + \sum_{k=1}^{2m}\alpha_k\mathbf{z}_k(\xi)=0 \quad \text{for all } \xi \in \mathbb{R}.$$
Clearly, it also holds for all $\xi\in\mathbb{R}$ that 
$$\alpha_0\mathbf{u}_*'(\xi) + \sum_{k=1}^{2m}\alpha_k\mathbf{z}_k'(\xi)=0.$$ 
In particular, for $\xi = 0$ and combining the above two equations, we get that 
$$\alpha_0U_*^{\perp}(0) + \sum_{k=1}^{2m}\alpha_kZ_k(0)=0.$$
But since we know that $\{W_k(0); k=1,...,2m\}\cup\{V_*^{\perp}(0)\}=\{Z_k(0); k=1,...,2m\}\cup\{U_*^{\perp}(0)\}$ is a basis for $\ker Q^*$, we must have that $\alpha_k = 0$ for all $k=0,...,2m$. And so $F_k'(0)$ are linearly independent.

Now, we consider the decomposition $X_{\beta}= \ker \overline{F}'(0) \oplus \Xi,$ where $\overline{F}(B) = \big(F_1(B),...,F_{2m}(B)\big)^T$ and $\Xi$ has dimension $2m$. Then any $B\in X_{\beta}$ can be expressed uniquely as the sum $B=B_1+ B_2$, where $B_1 \in \ker \overline{F}'(0)$ and $B_2\in \Xi$.

Now we define the function $f:\ker \overline{F}'(0) \times \Xi \to \mathbb{R}^{2m}$ by $f(B_1,B_2) = \overline{F}(B_1+B_2)=\overline{F}(B).$ This is clearly smooth since $\overline{F}$ is. Differentiating using the chain rule gives us that $\partial_{B_i}f(B_1,B_2)B_i=\overline{F}'(B_1+B_2)B_i$ which implies that  $\ker\partial_{B_1}f(0,0)=\ker \overline{F}'(0)$ and $\ker\partial_{B_2}f(0,0)= \{0\}$.
Next, we use the implicit function theorem so that $f(B_1,B_2)=0$ defines $B_2$ as a smooth function of $B_1$ in a neighbourhood of $B_1=0,B_2=0$. We denote this function by $g:U\subset \ker \overline{F}'(0) \to \Xi$, where $U$ is a neighbourhood of $0\in\ker \overline{F}'(0)$. Then $g(B_1) = B_2$ if and only if $f(B_1,B_2)=0$, or equivalently, if and only if $\overline{F}(B)=0$.
Further, let $\zeta: U \subset \ker \overline{F}'(0) \to \ker \overline{F}'(0) \times \Xi$ be defined by $B_1 \mapsto (B_1,f(B_1))$. The map is clearly smooth and injective, and thus invertible onto its image, $\zeta(U)$.

Defining $G(x,y)= y - \zeta(x)$, we can apply the implicit function theorem again on the equation $G(x,y)=0$, giving the smooth function $h(x)=y$, defined locally. This must then be $\zeta^{-1}$, and we are done. 
\end{proof}

\begin{proof}[Proof of theorem \ref{thm:second}]
We follow the proof of Theorem \ref{thm:main} until equation \eqref{eq:fkprim}.
Recall that Assumption \ref{assume:3} holds, and so $\mathbf{u}_*$ is non-zero only in the first entry. Then we can choose the corresponding initial conditions such that all $\mathbf{z}_k$ is orthogonal to $\mathbf{u}_*$, making all $\mathbf{z}_k$, and therefore $\mathbf{z}$, zero at all entries except the first. This results in the equation
    $$\int_{-\infty}^{\infty}\langle \mathbf{z}(\xi),B(\xi)\mathbf{u}_*(\xi)\rangle d\xi=0.$$
By Lemma \ref{lem:modcovlem} it now follows that $\mathbf{z}=0$.
By the same arguments as in the proof of Theorem \ref{thm:main} (with $\alpha_0=0$ by the ortogonality), we may conclude that $\alpha_k=0$ for all $k=1,...,2m$. And so, the result follows.
\end{proof}

\section{An Example}
This section serves to conclude the paper by providing an example of the investigated problem.

\begin{example}\label{ex:1}
Let
$$V_p(x) = 2\cos(2x).$$
This eigenvalue equation for the corresponding operator then becomes the famous Mathieu equation
$$u'' + \big(\lambda-2q\cos(2x)\big)u =0,$$
with $q = 1$.
It is known that periodic Schrödinger operators have a non-empty and purely continuous spectrum with a certain band structure \cite{east}. We choose $\lambda_0$ from one of these bands, such that it belongs to the spectrum of the operator
$$\mathcal{L}_p = -\dfrac{d^2}{dx^2} + V_p.$$
Now let 
$$V_c(x) = 1-\dfrac{2}{\cosh^2(x)} + \lambda_0.$$
Then $u_1(x) = 1/\cosh(x)$ is an eigenfunction of the operator
$$\mathcal{L}_c = -\dfrac{d^2}{dx^2} + V_c,$$
with corresponding eigenvalue $\lambda_0$.

We choose our coefficient matrix of the form

$$A(x) = \begin{bmatrix} V_c(x) & 0 \\
0  & V_p(x)
\end{bmatrix}.$$

\bigskip

This means that $\lambda_0$ is a simple embedded eigenvalue of $\mathcal{L} = -\frac{d^2}{dx^2} + A$, with corresponding eigenfunction $(u_1,0)^T$. The matrix $A$ is clearly asymptotically periodic, with 
$$A_p(x)=\begin{bmatrix}
    1+\lambda_0 & 0 \\
    0 & 2\cos(2x)
\end{bmatrix}.$$
Moreover, we have that 

$$\lVert A -A_p \rVert_{X_{\beta}} = 2\cdot\sup_{x\in\mathbb{R}}\dfrac{(1+\lvert x\rvert)^{\beta}}{\cosh^2(x)}<\infty, \text{ for all } \beta >1.$$
Thus, Assumption \ref{assume:1} and \ref{assume:2} are fulfilled.
For $\mathbf{z}_1$ and $\mathbf{z}_2$ we choose some generalized eigenfunctions

$$\mathbf{z}_1(x) = \begin{bmatrix} 0 \\
z_1(x)
\end{bmatrix} \quad \text{and} \quad \mathbf{z}_2(x) = \begin{bmatrix} 0 \\
z_2(x)
\end{bmatrix}.$$

Since 

$$B = \begin{bmatrix} b_{11}(x) & b_{12}(x) \\
b_{12}(x) & b_{22}(x)
\end{bmatrix}$$

and the fact that 

\begin{equation*}
    \begin{aligned}
\langle \mathbf{z}_k(\xi), B(\xi)\mathbf{u}_*(\xi) \rangle = \begin{bmatrix}
0 & z_k(\xi)
\end{bmatrix}\begin{bmatrix} b_{11}(\xi) & b_{12}(\xi) \\
b_{12}(\xi) & b_{22}(\xi)
\end{bmatrix}
\begin{bmatrix}
    u_1(\xi) \\
    0
\end{bmatrix} =
    b_{12}(\xi)z_k(\xi)u_1(\xi),
    \end{aligned}
\end{equation*}

we have that 

$$F_k'(0)B = \int_{-\infty}^{\infty}b_{12}(\xi)z_k(\xi)u_1(\xi) \quad \text{for } k=1,2.$$

Hence, the manifold $\mathscr{M}$ is tangent to the subspace of perturbations $B \in X_{\beta}$ such that the off-diagonal elements are orthogonal to $z_k(x)u_k(x)$. This follows since $\mathscr{M}$ is described in a neighbourhood of $B=0$ by the equations $F_k(B)=0$, and that the eigenvalue can only persist if $B\in \mathscr{M}$.
\end{example}

\section*{Acknowledgements}
The authors would like to thank the anonymous referee for careful reading of the paper and for useful feedback.

\printbibliography

\end{document}